\documentclass{amsart}
\usepackage[utf8]{inputenc}
\usepackage{amsmath,amsthm}
\usepackage{comment}
\usepackage{amssymb}
\usepackage{stmaryrd}
\usepackage[hidelinks]{hyperref}
\usepackage{color}
\usepackage[left=3cm,right=3cm,top=3cm,bottom=4.5cm]{geometry}
\usepackage{pdfsync}
\usepackage{tikz-cd}
\usepackage{url}
\usepackage{enumitem}
\hypersetup{linktoc=all}
\usepackage[capitalise,noabbrev]{cleveref}

\newcommand{\defterm}[1]{\emph{#1}}

\theoremstyle{plain}
	\newtheorem{thm}{Theorem}[section]
	\newtheorem*{thm*}{Theorem}
	\newtheorem{cor}[thm]{Corollary}
	\newtheorem*{cor*}{Corollary}
	
	\newtheorem*{prop*}{Proposition}
	\newtheorem{lem}[thm]{Lemma}
	\newtheorem*{lem*}{Lemma}
	
	\newtheorem*{ex*}{Exercise}
	
	\newtheorem*{claim*}{Claim}
	\newtheorem{question}[thm]{Question}
	\newtheorem*{question*}{Question}
	\newtheorem{fact}[thm]{Fact}
	\newtheorem*{fact*}{Fact}
\theoremstyle{definition}
	\newtheorem{Def}[thm]{Definition}
	\newtheorem*{Def*}{Definition}
	
	\newtheorem*{obs*}{Observation}
	\newtheorem{rmk}[thm]{Remark}
	\newtheorem*{rmk*}{Remark}
	
	\newtheorem{soln*}{Solution}
	
	\newtheorem*{note*}{Note}
	
	\newtheorem*{eg*}{Example}	
	
	\newtheorem*{construction*}{Construction}
	
	\newtheorem*{warning*}{Warning}
	
	\newtheorem*{conj*}{Conjecture}
	
\newcommand{\ints}{\mathbb{Z}}
\newcommand{\rats}{\mathbb{Q}}

\newcommand{\cpx}{\mathbb{C}}
\newcommand{\nats}{\mathbb{N}}

\newcommand{\id}{\mathrm{id}}

\newcommand{\Map}{\mathrm{Map}}

\newcommand{\Psh}{\mathsf{Psh}}
\newcommand{\Fun}{\mathrm{Fun}}

\DeclareMathOperator{\Ind}{Ind}

\newcommand{\Mod}{\mathsf{Mod}}

\newcommand{\calC}{\mathcal{C}}

\newcommand{\calL}{\mathcal{L}}

\newcommand{\Cyc}{\mathsf{Cyc}}
\newcommand{\fin}{\mathrm{fin}}

\newcommand{\catdeloop}{B}
\newcommand{\Paracyc}{\mathsf{PCyc}}
\newcommand{\inj}{\mathrm{inj}}
\newcommand{\pjv}{\mathbb{P}}
\newcommand{\DCyc}{\mathsf{DCyc}}
\newcommand{\Pos}{\mathsf{Pos}}

\begin{document}

\title{Non-trivial higher homotopy of first-order theories}
\author[T. Campion]{Tim Campion}

\author[J. Ye]{Jinhe Ye}

\date{\today}

\address{Department of Mathematics, John Hopkins University}
\email{tcampio1@jh.edu}

\address{Mathematical Institute, University of Oxford}
\email{jinhe.ye@maths.ox.ac.uk}
\begin{abstract}
    Let $T$ be the theory of dense cyclically ordered sets with at least two elements. We determine the classifying space of $\Mod(T)$ to be homotopically equivalent to $\mathbb{CP}^\infty$. In particular, $\pi_2(\lvert\Mod(T)\rvert)=\mathbb{Z}$, which answers a question in our previous work. The computation is based on Connes' cycle category $\Lambda$.
\end{abstract}

\maketitle

\tableofcontents

\section{Introduction}
In the previous work of the authors with Cousins~\cite{lascar-homotopy}, it was shown that given a complete first-order theory $T$, the fundamental group of the classifying space of $\Mod(T)$, the category of models of $T$ with elementary embeddings, is exactly the Lascar group. The Lascar group was introduced as a model-theoretic Galois group~\cite{lascar}. Other topological invariants of the classifying space $|\Mod(T)|$ beyond the fundamental group appear to be worthy of exploration. In all the examples appearing in~\cite{lascar-homotopy}, the homotopy groups in degrees 2 and above were observed to vanish. By relaxing the restriction to models of theories in other more expressive logics, for example, abstract elementary classes (AECs), arbitrary homotopy types can be realized, even for AECs with nice model-theoretic properties~\cite{higher-aec}. This motivates the following question:
\begin{question}\label{question}
Is it true that for every first-order theory $T$, $\pi_2(\lvert\Mod(T)\rvert)$ vanishes? If not, what is the model-theoretic content of it?
\end{question}
In this short note, we answer the first part of \cref{question} in the negative. Specifically, for $T$ the theory of dense cyclically ordered sets with at least two elements, $\lvert\Mod(T)\rvert\simeq \mathbb{CP}^\infty$. In particular, $\pi_2(\lvert \Mod(T) \rvert) = \ints \neq 0$. The computation is based on Connes' cycle category $\Lambda$~\cite{connes} and the work of Nikolaus and Scholze~\cite{nikolaus-scholze}. The theory of dense cyclically ordered sets is a canonical example in model theory where existence of non-forking extensions fails over $\varnothing$.  In forthcoming work, we will show that for $T$ stable, $\pi_2(\lvert\Mod(T)\rvert)$ is always trivial. Stability can be characterized axiomatically in terms of forking independence (see for example~\cite{harrington,kim-pillay}), and clever usage of these independence axioms will enable us to trivialize the second homotopy group. This seems to suggest that $\pi_2(\lvert\Mod(T)\rvert)$ is related to certain aspects of forking/dividing. However, the connection is still rather mysterious, see~\cref{rmk:boolean}. For an introduction to forking and dividing, see~\cite[Chapter 7]{tent-ziegler}.

The contents of this note are as follows. In \cref{sec:cyc}, we introduce the category of \emph{cyclically ordered sets}, and compute its homotopy type following \cite{nikolaus-scholze}. In \cref{sec:dense}, we introduce the theory of \emph{ dense cyclically ordered sets} and its category of models $\DCyc$. We show that it is homotopy equivalent to the category of cyclically ordered sets. Lastly, for a first-order theory $T$, we have a brief discussion of $\pi_2|\Mod(T)|$ with forking and dividing in model theory.
\paragraph{\bf{Acknowledgement.}} We would like to express our gratitude to the anonymous referee for feedback and suggestions for improvement. TC is grateful for the support of the ARO under MURI Grant W911NF-20-1-0082.

\section{Nonempty cyclically ordered sets}\label{sec:cyc}

In this section, we introduce the category of \emph{cyclically ordered sets} (\cref{def:cycord}), and compute the homotopy type of the category thereof (\cref{cor:cyc-bs1-bis}).

\subsection{Cyclically ordered sets}

We introduce cyclically ordered sets from a model-theoretic perspective (\cref{def:cycord}) and show that this category may alternatively be defined via an Ind construction (\cref{lem:acc}).

\begin{Def}\label{def:cycord}
The \defterm{language of cyclically ordered sets} $
\mathcal{L}_\Cyc$ is the language with a ternary relation symbol $R(x,y,z)$. A \defterm{cyclically ordered set} is an $\mathcal{L}_\Cyc$-structure $(S,R(x,y,z))$ such that the following holds:
\begin{enumerate}[align=left]
    \item [Asymmetry:] $\forall x,y,z,w$ $R(x,y,z)\wedge R(x,z,w)\rightarrow y\neq w$;
    \item [Transitivity:] $\forall x,y,z,w$ $R(x,y,z)
    \wedge R(x,z,w) \rightarrow R(x,y,w)$;
    \item [Connectedness:] For all distinct $x,y,z$, $R(x,y,z)$ or $R(z,y,x)$;
    \item [Cyclicity:] $\forall x,y,z$ $R(x,y,z)\rightarrow R(y,z,x)$.
\end{enumerate}

An \defterm{embedding of cyclically ordered sets} (or in short, embedding) is an $\mathcal{L}_\Cyc$-embedding. Namely $f:C\to D$ is an embedding of cyclically ordered sets if $f$ is injective and $R_C(a,b,c)\leftrightarrow R_D(f(a),f(b),f(c))$. Note that the preservation of $R$ implies $f$ is injective once $C$ has at least $3$ elements. We introduce notation for the following categories of cyclically ordered sets:
\begin{itemize}
    \item $\Cyc$ is the category of cyclically ordered sets and embeddings between them.
    \item $\Cyc_{>0} \subset \Cyc$ is the full subcategory of nonempty cyclically ordered sets.
    \item $\Cyc^\fin_{>0} \subset \Cyc_{>0}$ is the full subcategory of finite, nonempty cyclically ordered sets.
\end{itemize}
\end{Def}
\begin{rmk}
    It is worth pointing out that the axioms other than Cyclicity for a cyclically ordered set states that $R(x,y,z)$ is a strict linear ordering on the complement of $x$ for each $x$.
\end{rmk}

\begin{lem}\label{lem:acc}
$\Cyc_{>0}$ is an $\aleph_0$-accessible category. The finitely-presentable objects are the finite cyclically ordered sets. Thus we have an equivalence $\Cyc_{>0} \simeq \Ind(\Cyc^\fin_{>0})$. 
\end{lem}
\begin{proof}
It is clear that directed colimits are given by taking unions. Thus the finite nonempty cyclically ordered sets are finitely-presentable. As every nonempty cyclically ordered set is the (directed) union of its finite cyclically ordered subsets, the claim follows.
\end{proof}

\begin{rmk}
The above mentioned categories $\Cyc_{>0}$ and $\Cyc^\fin_{>0}$ are closely related to the \emph{cyclic sets} of \cite{connes}, but there are some differences. The first difference is the site of definition: Connes' cycle category $\Lambda$ is a Reedy category with both ``monic" and ``epic" morphisms; our category $\Cyc_{>0}^\fin \subset \Lambda$ contains only the monomorphisms, as is more natural in the setting of model theory. The other main difference is that Connes' category of cyclic sets is a presheaf category $\Psh(\Lambda)$, whereas our category of cyclically ordered sets is just the Ind-category $\Cyc_{>0} = \Ind(\Cyc_{>0}^\fin) \subset \Psh(\Cyc_{>0}^\fin)$ (\cref{lem:acc}).

Following this observation, we denote by $[n-1]$ the cyclically ordered set with $n$ elements.
\end{rmk}

\subsection{Finite cyclically ordered sets and paracycles}\label{subsec:paracyclic}

Recall (cf. \cite[Appendix B]{nikolaus-scholze}) that Connes' cycle category $\Lambda$ may be defined as the quotient of the paracycle category (variously denoted $\mathsf L$ or $\Lambda_\infty$) by a free categorical action of the categorical group $\catdeloop \ints$. In this subsection, we check that $\Cyc_{>0}^\fin$ is likewise the quotient of the category $\Paracyc^\fin_{>0}$ of monomorphisms in $\Lambda_\infty$ by the restriction of this categorical action.

\begin{rmk}
    For us, a \defterm{categorical group} $G$ is a strict group object in the category of categories. An \defterm{action} of a categorical group $G$ on a category $C$ is meant in the strict sense -- a strictly unital and associative functor $G \times C \to C$.
    
    We let $B\ints$ denote the categorical delooping of the integers. This is the category with one object, with morphism set $\ints$, and composition given by addition of integers. This category is a categorical group, i.e. a strict group object in the category of categories, in a unique way. As in \cite[Appendix B]{nikolaus-scholze}, a $\catdeloop \ints$ action on $C$ amounts to a $\ints$ action on the homsets of $C$. 
\end{rmk}

\begin{Def}[\cite{cisinski,nikolaus-scholze}]\label{def:paracyc}
A \defterm{graded poset} is a poset equipped with an automorphism. A \defterm{paracycle} is a graded poset of the form $(1/n)\ints$ for some $n \geq 1$, where the automorphism is given by $x \mapsto x + 1$. We denote by $\Paracyc^\fin_{>0}$ the category of paracycles and poset embeddings which respect the automorphism. 

The category of graded posets admits an action by the categorical group $\catdeloop \ints$, which adds the shift automorphism to any morphism. This restricts to a free action on $\Paracyc^\fin_{>0}$. There is a canonical functor $\Paracyc^\fin_{>0} \to \Cyc^\fin_{>0}$ sending $(1/n)\ints$ to $[n-1]$.
\end{Def}

\begin{rmk}
The \defterm{paracycle category}, denoted $\mathsf{L}$ by \cite{cisinski} and $\Lambda_\infty$ by \cite{nikolaus-scholze}, is the category with the same objects as $\Paracyc^\fin_{>0}$ but with morphisms \emph{all} order-preserving, automorphism-respecting maps. Connes' cycle category $\Lambda$ is then defined to be the quotient of $\Lambda_\infty$ by the canonical $\catdeloop \ints$-action.
\end{rmk}

\begin{rmk}
Similarly to the case of $\Lambda_\infty$ (\cite[Appendix B]{nikolaus-scholze}) $\Paracyc^\fin_{>0}$ is a full subcategory of the functor category $\Fun(\catdeloop \ints, \Pos)$, where now $\Pos$ is the category of posets and embeddings between them. As in that case, the action of $\catdeloop \ints$ on $\Paracyc^\fin_{>0}$ is inherited from the action of the categorical group $\catdeloop \ints$ on itself.
\end{rmk}

Just as the functor $\Lambda_\infty \to \Lambda$ exhibits $\Lambda$ as the quotient of $\Lambda_\infty$ by the action of $\catdeloop \ints$, the same is true of the functor $\Paracyc^\fin_{>0} \to \Cyc^\fin_{>0}$:

\begin{lem}\label{lem:horb}
The canonical functor $\Paracyc^\fin_{>0} \to \Cyc^\fin_{>0}$ exhibits $\Cyc^\fin_{>0}$ as the quotient of $\Paracyc^\fin_{>0}$ by the $\catdeloop \ints$ action. This quotient is preserved by the nerve functor. 
\end{lem}
\begin{proof}
As in \cite[Appendix B]{nikolaus-scholze}, the $\catdeloop \ints$ action on $\Paracyc^\fin_{>0}$ amounts to a $\ints$ action on the homsets of $\Paracyc^\fin_{>0}$, given by applying the shift map of the source (equivalently, of the target) object. Moreover, the quotient by the $\catdeloop \ints$ action is given by quotienting each homset by this $\ints$ action. The map $\Map_{\Paracyc^\fin_{>0}}((1/m)\ints, (1/n)\ints) \to \Map_{\Cyc^\fin_{>0}}([m-1], [n-1])$ is clearly surjective. Moreover, two morphisms $f,g$ are identified under this map if and only if they differ by a shift. Thus this map quotients by the $\ints$ action as desired. As a quotient which is the identity on objects and full, this quotient is preserved by the nerve functor.
\end{proof}

\begin{rmk}
We might enlarge the category $\Cyc_{>0}^\fin$ to allow for certain non-injective maps. One way to do this would be to allow for \emph{monotone} maps, i.e. $f:C\to D$ such that $R_D(f(x), f(y), f(z))\rightarrow R_C(x,y,z)$. However, this would \emph{not} result in a category equivalent to Connes' category $\Lambda$. For example, there is a unique monotone map $[n] \to [0]$ for each finite cyclically ordered set $[n]$, but in $\Lambda$ there are $n+1$ distinct morphisms $[n] \to [0]$. It is only the monomorphisms of $\Lambda$ which are well-described as maps of structures.
\end{rmk}

\subsection{The homotopy type of nonempty cyclic sets}

In this subsection, we show that the homotopy types of $\Cyc_{>0}^\fin$ and $\Cyc_{>0}$ are $BS^1 = \cpx \pjv^\infty$. The idea is to show that the homotopy type of $\Paracyc_{>0}^\fin$ is contractible and use that $\Cyc_{>0}^\fin$ is the (homotopy) orbits of the latter by the free action of $S^1 = B \ints$ from \cref{def:paracyc}. This method of proof is well-known (\cite{cisinski}, \cite{nikolaus-scholze}) in the case of $\Lambda$ and $\Lambda_\infty$; we verify here that the same proof works when we restrict to categories of monomorphisms.

We begin by showing that $\Paracyc_{>0}^\fin$ is contractible by comparing it to the semisimplex category $\Delta^\inj$ of nonempty finite linear orders and embeddings between them. The analog of the following lemma for the inclusion $\Delta \to \Lambda_\infty$ of the simplex category into the paracycle category is well-known. The proof for the categories of monomorphisms is the same (although there is an error in the proof of \cite[Theorem B.3]{nikolaus-scholze}, since the intersection $\calC_{[t,t+1)} \cap \calC_{[s,s+1)}$ is often empty):

\begin{lem}[{\cite[8.5.15]{cisinski},\cite[Theorem B.3]{nikolaus-scholze}}]\label{lem:delta-paracyc}
The canonical functor $\Delta^\inj \to \Paracyc^\fin_{>0}$ is homotopy initial (i.e. it satisfies the dual hypotheses of Quillen's Theorem A) and in particular induces a homotopy equivalence of classifying spaces.
\end{lem}
\begin{proof}
For Quillen's Theorem A, see \cite[Theorem A]{quillen}. Let $(1/n)\ints \in \Paracyc^\fin_{>0}$, and consider the slice category $\Delta^\inj \downarrow (1/n)\ints$. For any $a,b\in (1/n)\ints$ with $a < b$, consider the full subcategory $\calC_{[a,b)}$ spanned by those maps $f : (1/k)\ints \to (1/n)\ints$ such that the image of $\{0,1/k,\dots,(k-1)/k\}$ is contained in the half-open interval $[a,b)$. Observe that every $f \in \calC$ is contained in $\calC_{[a,a+1)}$ for some $a \in (1/n)\ints$. We will induct on $b-a$ to show that $\calC_{[a,b)}$ is contractible. In the base case, if $b \leq a+1$, then $\calC_{[a,b)}$ is isomorphic to $\Delta^\inj \downarrow [m]$ where $m = n(b-a)$.  So $\calC_{[a,b)}$ has a terminal object and hence is contractible. For the inductive step, for any $a < b-1$, the square

\begin{equation*}
\begin{tikzcd}
\calC_{[b-1,b - 1/n)} \ar[r] \ar[d] & \calC_{[b-1,b)} \ar[d] \\
\calC_{[a,b - 1/n)} \ar[r] & \calC_{[a,b)}
\end{tikzcd}
\end{equation*}
is both a pushout and pullback, even after taking nerves, and it comprises cofibrations of simplical sets. Thus it is a homotopy pushout. By induction it is a homotopy pushout of contractible simplicial sets, and hence $\calC_{[a,b)}$ is contractible as well, completing the induction. Passing to the (filtered) colimit, $\cup_{a<b} \calC_{[a,b)} = \Delta^\inj \downarrow (1/n)\ints$, this category is weakly contractible as desired.
\end{proof}

\begin{cor}[\cite{cisinski,nikolaus-scholze}]\label{cor:paracyc-contr}
The category $\Paracyc^\fin_{>0}$ is weakly contractible.
\end{cor}
\begin{proof}
This follows from \cref{lem:delta-paracyc} and the fact that $\Delta^\inj$ is weakly contractible (\cite[1.7.24]{maltsiniotis}).
\end{proof}

\begin{cor}[\cite{cisinski,nikolaus-scholze,connes}]\label{cor:cyc-bs1}
The category $\Cyc^\fin_{>0}$ has the homotopy type of $BS^1 = K(\ints,2) = \cpx\pjv^\infty$.
\end{cor}
\begin{proof}
This follows from \cref{cor:paracyc-contr} and \cref{lem:horb}, since $\catdeloop\ints$ is equivalent as a topological group to $S^1$.
\end{proof}

\begin{cor}\label{cor:cyc-bs1-bis}
    The category $\Cyc_{>0}$ has the homotopy type of $BS^1 = K(\ints,2) = \cpx \pjv^\infty$.
\end{cor}
\begin{proof}
    This follows from \cref{cor:cyc-bs1} and \cref{lem:acc}, since for any category $\calC$ we have $|\calC| \simeq |\Ind(\calC)|$ (this follows from Quillen's Theorem A as for any $X \in \Ind(\calC)$, the category $\calC \downarrow X$ is filtered and hence contractible).
\end{proof}

\section{Dense cyclically ordered sets}\label{sec:dense}
In this section, we introduce the theory of dense cyclically ordered sets (\cref{def:dense}). We then compute the homotopy type of the category of models thereof (\cref{cor:maincor}).

\subsection{Completeness and quantifier elimination}

\begin{Def}\label{def:dense}
Let $T$ be the theory, in the language of cyclically ordered sets $\calL_\Cyc$, of a dense cyclically ordered set with at least two elements. More explicitly, the density condition is expressed as the following:
\[
\forall x,z \, (x\neq z\rightarrow \exists y \,R(x,y,z)).
\]
\end{Def}
The following follows from the back-and-forth argument in dense linear orders without endpoints.
\begin{fact}\label{fact:qe}
    $T$ has a unique countable model up to isomorphism, which is $\rats/\ints$ with the obvious cyclic ordering given by $<$ in $\rats$. Moreover, $T$ admits quantifier elimination in $\mathcal{L}_\Cyc$. 
\end{fact}
We will use $\DCyc$ to denote the category of models of $T$ where the morphisms are elementary embeddings. By quantifier elimination, the arrows are exactly the $\calL_\Cyc$-embeddings. In other words, $\DCyc$ is the full subcategory of $\Cyc_{>0}$ whose objects consist of the dense ones.

\subsection{The homotopy type of dense cyclically ordered sets}

\begin{Def}
Given a cyclically ordered set $C$, let $T(C) = C \times \{0,1\}$. We identify $C$ with $C \times \{0\} \subseteq T(C)$; call the inclusion map $\iota_C: C \to T(C)$. Give $T(C)$ the cyclic structure where the ordering is such that $(c,1)$ is placed in the cut immediately after $c$ with respect to $C$. More precisely, this means that $R((x,0),(x,1),y)$ for any $y$ distinct from $(x,0)$, $(x,1)$. If $f:C\to D$ is an embedding of cyclically ordered sets, $T(f):T(C)\to T(D)$ is the map that sends $(c,i)$ to $(f(c),i)$ for $i=0,1$.
\end{Def}
The following lemma is immediate.
\begin{lem}\label{lem:T}
$T(C)$ is a cyclically ordered set, and $C \to T(C)$ is an embedding of cyclically ordered sets. The construction $C \mapsto T(C)$ is functorial, and the embedding $C \to T(C)$ is natural. Moreover, if $C$ is nonempty then $T(C)$ has at least two elements.
\end{lem}

\begin{Def}
For a cyclically ordered set $C$, define $T^\infty(C) = \cup_n T^n(C)$ where the inclusion maps are given by $\iota$; continue to call the inclusion $\iota: C \to T^\infty(C)$. $T^\infty(f)$ is also defined accordingly.
\end{Def}

\begin{lem}\label{lem:Tinf}
 $T^\infty(C)$ is a cyclically ordered set, and $C \to T^\infty(C)$ is an embedding of cyclically ordered sets. The construction $C \mapsto T^\infty(C)$, $f\mapsto T^\infty(f)$ is functorial, and the embedding $\iota: C \to T^\infty(C)$ is natural. Moreover, $T^\infty(C)$ is an object in $\DCyc$ if $C$ is nonempty.
\end{lem}
\begin{proof}
 The naturality follows from~\cref{lem:T}. So it suffices to show that $T^\infty(C)$ is dense. Let $x \neq y \in T^\infty(C)$. By construction, there is $n\in\nats$ such that $x, y \in T^n(C)$. If there is $z \in T^n(C)$ with $R(x,z,y)$, then we are done. Otherwise, observe that in $T^{n+1}(C)$, $R((x,0), (x,1), (y,0))$ holds.
\end{proof}

\begin{cor}\label{cor:fraisse-heq}
The functor $T^\infty : \Cyc_{>0} \to \DCyc$ and the inclusion functor $i : \DCyc \to \Cyc_{>0}$ induce inverse homotopy equivalences after passing to classifying spaces.
\end{cor}
\begin{proof}
We have natural transformations $\iota: \id \Rightarrow iT^\infty$ and $\iota: \id \Rightarrow T^\infty i$.
\end{proof}

\begin{cor}\label{cor:maincor}
We have $|\DCyc| \simeq BS^1 = K(\ints,2) = \cpx \pjv^\infty$. In particular, $\DCyc$ is an elementary class whose classifying space is not aspherical.
\end{cor}
\begin{proof}
$BS^1 = |\Cyc_{>0}| \simeq |\DCyc|$ by \cref{cor:cyc-bs1-bis} and \cref{cor:fraisse-heq} respectively.
\end{proof}

\begin{rmk}
If we tried to include the empty cyclically ordered set into $\Cyc^\fin_{>0}$, the functor $T^\infty$ would not turn the empty one into a nonempty one. So we wouldn't be able to relate our category homotopically to a category where emptiness is ruled out by the theory.
\end{rmk}
\begin{rmk}\label{rmk:boolean}
Dense cyclically ordered sets are canonical examples where the $0$-definable set $x=x$ forks over $\varnothing$. However, the exact relationship between forking and the higher homotopy remains mysterious. Another canonical example of forking $\neq$ dividing comes from the theory of infinite atomic boolean algebras $T$. It has quantifier elimination in $\mathcal{L}=\{\wedge,\vee,^c,0,1,(A_n)_{n\in \nats_{>0}}\}$ where each $A_n$ is to be interpreted as the predicate for $n$-atoms: the join of $n$ distinct atoms. An element  $a$ is \defterm{infinite} if it is above infinitely many atoms, and \defterm{coinfinite} if $a^c$ is infinite. Note that any $a,b$ that are infinite and coinifinite have the same type over $\varnothing$. The predicate $A_1(x)$ does not divide but forks since one can find an infinite and coinfinite element $a$ $A_1(x)\rightarrow x\wedge a=x$ or $x\wedge a^c=x$ and the formula $x\wedge a=x$ divides for any infinite and coinfinite $a$. On the other hand, the classifying space of the category of models of infinite atomic boolean algebra $\Mod(T)$ is contractible. To see this, note that there is a functor $F$ from the category of infinite sets with embeddings to $\Mod(T)$ by sending each set $A$ to $\mathcal{P}^\fin(A)$ and associate to each map $f:A \to B$ the map $\mathcal{P}^\fin(A)\to \mathcal{P}^\fin(B)$ induced by $f$ on the atoms. There is also the ``forgetful" functor $U$ from $\Mod(T)$ to the category of infinite sets by sending each atomic boolean algebra to its set of atoms. There are obvious natural transformations from the identity to $UF$, and from the identity to $FU$, so and thus the classifying spaces of these categories are homotopy equivalent.

\end{rmk}

\bibliographystyle{alpha}
\bibliography{bibliography}

\end{document}